\documentclass[a4paper,10pt]{article}
\usepackage{graphicx,fancyhdr,array,multicol,multirow,enumitem}
\usepackage[margin=1.5cm,vmargin={1.5cm,1.5cm},includefoot]{geometry}
\usepackage{fancybox}
\usepackage{xcolor}
\usepackage{cancel}
\usepackage[normalem]{ulem}
\usepackage{tikz}

\usepackage[utf8]{inputenc}
\usepackage[T1]{fontenc}
\usepackage[english]{babel}
\usepackage{amssymb,mathtools,amsfonts,amsthm}
\usepackage{stmaryrd}
\usepackage{hyperref}
\usepackage{cite} 
\newtheorem{definition}{Definition}
\newtheorem{lemma}{Lemma}
\newtheorem{theorem}{Theorem}

\newtheorem{remark}{Remark}

\newtheorem{assumption}{Assumption}

\newcommand{\N}{\mathbb N}

\newcommand{\R}{\mathbb R}

\newcommand{\Z}{\mathbb Z}

\newcommand{\cE}{\mathcal E}
\newcommand{\cF}{\mathcal F}
\newcommand{\cL}{\mathcal L}

\title{Stationary measures for the Porous Medium Model}
\author{Oriane Blondel}

\begin{document}
\maketitle
\begin{abstract}{We study the stationary measures for variants of the Porous Medium Model in dimension 1. These are exclusion processes that belong to the class of kinetically constrained models, in which an exchange can occur between $x$ and $x+1$ only if there is a particle either at $x-1$ or $x+2$. We show that any stationary probability measure can be decomposed into a frozen part and a mixture of product measures (although there exist invariant sets which have zero probability under these measures). The proof adapts entropy arguments from \cite{holley-stroock,liggett}.}\end{abstract}

\section{Introduction}

The Porous Medium Model (PMM), is the process on $\Omega:=\{0,1\}^\Z$ in which a particle jumps between $x$ and $x+1$ or $x+1$ and $x$ at rate $\eta(x-1)+\eta(x+2)$. In particular, isolated particles (at distance at least $3$ from any other) are frozen, and this means the process admits infinitely stationary measures. 
This model has been introduced in \cite{bertini-toninelli}. It takes its name from the fact (established in \cite{GLT}) that in the hydrodynamic limit, for initial densities bounded away from $0$ and $1$, the density of the system converges to the solution of the Porous Medium Equation with $m=2$:
\begin{equation}
 \partial _t\rho=\Delta(\rho^2).
\end{equation}

We are interested in the study of the stationary probability measures for this model. Apart from this being a very natural question, such knowledge can be instrumental in the study of the hydrodynamic limit of the system \cite{Funaki,stefan}. 

It is well known that for any $\rho\in[0,1]$, $\mu_\rho:=\underset{x\in\Z}{\otimes}\mathrm{Ber}(\rho)$ is a reversible measure for this dynamics. We also noticed already that there are infinitely stationary measures concentrated on frozen configurations. Our main result is that, although there exist invariant sets that contain no frozen configuration and have $\mu_\rho$--probability $0$ for any $\rho$ (see \eqref{e:defF'}--\eqref{e:defE'}), the only extremal stationary probability measures are either concentrated on frozen configurations, or product.

The strategy of the proof is to use entropy arguments introduced in \cite{holley-stroock} to argue that any stationary measure concentrated on non-frozen configurations is in fact reversible w.r.t.\@ allowed jumps. That is the content of Lemma~\ref{l:exchangeinvariant}, which is proved in Section~\ref{s:Lemma}. The rest of the proof consists in exploiting this result to show that any putative stationary measure gives the same probability to certain sets of sub-configurations. That is the content of Section~\ref{s:proofTheorem}. In turn, these constructions rely on the identification of allowed transformations of configurations, collected in Section~\ref{s:Facts}.

\section{Model and result}
\subsection{Notations and definitions}
We consider interacting particle systems with state space $\Omega:=\{0,1\}^\Z$. As usual, for $\eta\in\Omega$, we say there is a particle at $x$ if $\eta(x)=1$, and that $x$ is empty if $\eta(x)=0$. We will also need to consider configurations restricted to subsets of $\Z$. For $\Lambda$ a subset of $\Z$, we define $\Omega_\Lambda=\{0,1\}^\Lambda$. For $\eta\in\Omega$, $\eta_{|\Lambda}$ is the element of $\Omega_\Lambda$ such that for all $x\in\Lambda$, $\eta_{|\Lambda}(x)=\eta(x)$.  Recall that a function $f:\Omega\rightarrow\R$ is \emph{local} if there exists $\Lambda$ finite subset of $\Z$ such that: $\forall \eta,\eta'\in\Omega$, $\eta_{|\Lambda}=\eta'_{|\Lambda}\Rightarrow f(\eta)=f(\eta')$. The 
support of a local function is the smallest such $\Lambda$. For $\sigma\in\Omega_\Lambda$ and $\nu$ a measure on $\Omega$, we denote abusively $\mathbf{1}_\sigma$ the indicator function of $\{\eta\in\Omega\ \colon\ \eta_{|\Lambda}=\sigma\}$ and we write $\nu(\sigma)=\nu(\mathbf{1}_\sigma)$. If a function $f$ has support in $\Lambda$ and $\sigma\in\Omega_\Lambda$, we write $f(\sigma)$ for the common value of $f(\eta)$, with $\eta\in\Omega$ such that $\eta_{|\Lambda}=\sigma$. Also, for $\sigma\in\Omega_\Lambda$, we write $|\sigma|=\sum_{x\in\Lambda}\sigma(x)$. For $n\in\N$, $\Lambda_n:=[-n,n]\cap\Z$. For $\rho\in (0,1)$, $\mu_\rho=\underset{x\in\Z}{\otimes}\mathrm{Ber}(\rho)$ is the product probability measure on $\Omega$ with homogeneous density $\rho$.

\bigskip

To define the models we consider, we consider a family of constraints $(c_x(\eta))_{x\in\Z,\eta\in\Omega}$ that satisfies the following assumption.
\begin{assumption}\label{ass}The family $(c_x(\eta))_{x\in\Z,\eta\in\Omega}$ takes values in $\R_+$ and has the following properties:
	\begin{itemize}
	 \item translation invariance: for any $x\in\Z$, $\eta\in\Omega$, $c_x(\eta)=c_0(\eta(x+\cdot))$; 
	 \item $c_0$ is a local function;
	 \item for all $\eta\in\Omega$, $c_0(\eta)=c_0(\eta^{0,1})$;
	 \item for all $\eta\in\Omega$, $c_0(\eta)>0$ iff $\eta(-1)+\eta(2)>0$.
	\end{itemize}
	\end{assumption}

For $f$ local function, $\eta\in\Omega$, let us define
\begin{equation}\label{e:defL}
 \cL f(\eta)=\sum_{x\in\Z}c_x(\eta)\left[f(\eta^{x,x+1})-f(\eta)\right],
\end{equation}
where $\eta^{x,x+1}(y)=\eta(y)$ if $y\notin\{x,x+1\}$, $\eta^{x,x+1}(x)=\eta(x+1), $ $\eta^{x,x+1}(x+1)=\eta(x)$.

Let us introduce some additional terminology. 

\begin{definition} Let $\Lambda$ be an interval of $\Z$. 
 \begin{enumerate}
  \item  For $\eta\in\Omega$, the transformation $\eta\mapsto\eta^{x,x+1}$ is an \emph{allowed jump} if $c_x(\eta)>0$.
  \item For $\{x,x+1\}\subset\Lambda$, $\sigma\in\Omega_\Lambda$, $\sigma\mapsto\sigma^{x,x+1}$ is an \emph{allowed jump inside $\Lambda$} if $c_x(0\cdot\sigma)>0$, where $0\cdot\sigma$ is the configuration equal to $0$ on $\Lambda^c$ and to $\sigma$ on $\Lambda$.
  \item For $\eta\in\Omega$, $x\in\Z$ such that $\eta(x)=1$, the particle at $x$ is \emph{active} if $c_x(\eta)+c_{x-1}(\eta)>0$ (equivalently, if there is another particle within distance $2$ of $x$).
  \item A pair of particles within distance $2$ of each other an \emph{mobile cluster}.
  \item A configuration without active particles is called \emph{frozen}.
 \end{enumerate}

\end{definition}
The reason for the name ``mobile cluster'' is that (as one can easily check), such a pair can move through space autonomously: no matter what the rest of the configuration is, for any target pair of neighbor or next-to-neighbor positions, there exists a sequence of allowed jumps that transports the pair from its initial position to the target (see Figure~\ref{f:mobile-cluster}). Note that the notions in this definition do not depend on the specific choice of the constraints satisfying the last condition in Assumption~\ref{ass}.

\begin{figure}
 \begin{center}
\begin{tikzpicture}[scale=0.65]
			\draw[thick] (-9.5,0) -- (-5.5,0);
			\foreach \x in {-9,...,-6}
			\draw (\x, -0.1) -- (\x,0.1);
			\fill (-9,0.6) circle(0.35);
			\draw[stealth-stealth, thick] (-9,1) to[bend left] (-8,1);
			\fill (-7,0.6) circle(0.35);
			
			\draw[thick] (-4.5,0) -- (-0.5,0);
			\foreach \x in {-4,...,-1}
			\draw (\x, -0.1) -- (\x,0.1);
			\fill (-3,0.6) circle(0.35);
			\fill (-2,0.6) circle(0.35);
			\draw[stealth-stealth, thick] (-2,1) to[bend left] (-1,1);
			
			\draw[thick] (1.5,0) -- (5.5,0);
			\foreach \x in {2,...,5}
			\draw (\x, -0.1) -- (\x,0.1);
			\fill (3,0.6) circle(0.35);
			\fill (5,0.6) circle(0.35);
			\draw[ stealth-stealth, thick] (3,1) to[bend left] (4,1);
			
			\draw[thick] (6.5,0) -- (10.5,0);
			\foreach \x in {7,...,10}
			\draw (\x, -0.1) -- (\x,0.1);
			\fill (9,0.6) circle(0.35);
			\fill (10,0.6) circle(0.35);
		
%
%
			
			\begin{scope}[ yshift=-3.3cm] 
			
			\draw[thick] (-11.5,0) -- (-7.5,0);
			\foreach \x in {-11,...,-8}
			\draw (\x, -0.1) -- (\x,0.1);
			\fill (-11,0.6) circle(0.35);
						\fill (-10,0.6) circle(0.25);
			\draw[stealth-stealth,thick] (-11,1) to[bend left] (-10,1);
			\fill (-9,0.6) circle(0.35);
			
			\draw[thick] (-6.5,0) -- (-2.5,0);
			\foreach \x in {-6,...,-3}
			\draw (\x, -0.1) -- (\x,0.1);
			\fill (-5,0.6) circle(0.35);
									\fill (-6,0.6) circle(0.25);
			\fill (-4,0.6) circle(0.35);
			\draw[stealth-stealth, thick] (-4,1) to[bend left] (-3,1);
			
			\draw[thick] (-1.5,0) -- (3.5,0);
			\foreach \x in {0,...,3}
			\draw (\x, -0.1) -- (\x,0.1);
			\fill (0,0.6) circle(0.25);
			\fill (1,0.6) circle(0.35);
			\fill (3,0.6) circle(0.35);
			\draw[stealth-stealth, thick] (1,1) to[bend left] (2,1);
			
			\draw[thick] (4.5,0) -- (8.5,0);
			\foreach \x in {5,...,8}
			\draw (\x, -0.1) -- (\x,0.1);
			\fill (7,0.6) circle(0.35);
			\fill (8,0.6) circle(0.35);
			\fill (5,0.6) circle(0.25);
			\draw[stealth-stealth, thick] (5,1) to[bend left] (6,1);
			
			\draw[thick] (9.5,0) -- (13.5,0);
			\foreach \x in {10,...,13}
			\draw (\x, -0.1) -- (\x,0.1);
			\fill (12,0.6) circle(0.35);
			\fill (13,0.6) circle(0.35);
			\fill (11,0.6) circle(0.25);
%
%
			\end{scope}
		\end{tikzpicture}  \end{center}
 \caption{On the first line, a strategy for moving a mobile cluster toward the right. The rest of the configuration is represented as empty but is in fact irrelevant. On the second line, steps 2--5, the beginning of the procedure that can be used to transport to the right an extra particle (depicted with smaller size) with a mobile cluster. The first step shows that we can assume the extra particle is not inside the mobile cluster.}\label{f:mobile-cluster}
\end{figure}
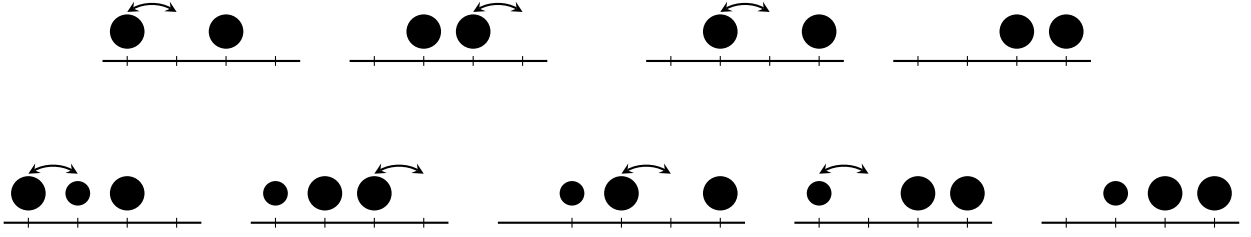

\bigskip

We are interested in the stationary measures of the process defined by~\eqref{e:defL}, that is in the set of probability measures $\nu$ on $\Omega$ such that $\nu(\cL f)=0$ for all local functions $f$. Note that the third condition in Assumption~\ref{ass} ensures that this process is reversible w.r.t.\@ $\mu_\rho$ for any $\rho\in (0,1)$.

	We first identify the invariant sets of the dynamics defined by $\eqref{e:defL}$. Notice that the number of particles is preserved by this dynamics. We define $\cF$ the set of frozen configurations, $\cF'$ the set of non-frozen configurations with finitely many particles, $\cF''$ the set of non-frozen configurations with finitely many holes, $\cE'$ the set of configurations with finitely many active particles that are in no other set. Finally, $\cE$ is what we call the ergodic set of configurations and is the complement of the previous sets. It is also a set with full measure under $\mu_\rho$, $\rho\in (0,1)$.
\begin{align}
	\mathcal{F} &= \{\eta\in\Omega\ \colon \ \forall x\in\Z,\ \eta(x)=1\Rightarrow \eta(x+1)+\eta(x+2)=0\},\label{e:defF}\\
	\cF'&=\cup_{k=2}^\infty\cF'_k, \text{ where }\mathcal{F}'_k=\{\eta\in\mathcal{F}^c\ \colon \  \sum_{x\in\Z}
	\eta(x)=k\}\ , \label{e:defF'}\\
	\cF''&=\cup_{k=0}^\infty\cF''_k,\text{ where }\cF''_k=\{\eta\in\cF^c\ \colon\ \sum_{x\in\Z}(1-\eta(x))=k\}\ ,\label{e:defF''} \\
	\cE'&=\{\eta\in\Omega\setminus\cF'\ \colon\ 0<\sum_{x\in\Z}(\eta(x)\eta(x+1)+\eta(x)\eta(x+2))<\infty\},\label{e:defE'}\\
	\cE&=\Omega\setminus(\cF\cup\cF'\cup\cF''\cup\cE').\label{e:defE}
	\end{align}
 We will also consider $G_\Lambda$ the set of ``good'' configurations that contain at least one mobile cluster in $\Lambda$, i.e.
 \begin{equation}
  G_\Lambda=\{\sigma\in\Omega_{\Lambda}\ \colon\ \exists \{x,y\}\subset\Lambda \text{ s.t. }|x-y|\in\{1,2\}\text{ and } \eta(x)+\eta(y)=2\}.
 \end{equation}
$G_\Lambda$ is invariant under allowed jumps inside $\Lambda$ (recall that we assume empty boundary condition). Finally, for $n\in\N^*$, let $G_n=G_{\Lambda_n}$.

	\subsection{Main result}

	Recall that $\mu_\rho$ is the product measure on $\Omega$ with density $\rho$. We aim to prove the following. 

\begin{theorem}\label{t:invariant}
	Under Assumption~\ref{ass}, any stationary measure $\nu$ for the process with generator $\cL$ can be decomposed into
	\begin{equation}\label{e:decomposition}
	\nu=\alpha_\cF\nu_{\cF}+\alpha_\cE\nu_\cE,
	\end{equation}
	with $\nu_X$ a stationary measure such that $\nu_X(X)=1$, $\alpha_X=\nu(X)$ for $X\in\{\cF,\cE\}$. Moreover, if $\nu(\cE)>0$, then there exists $\lambda$ a probability measure on $(0,1)$ such that 
	\begin{equation}\label{e:decompnuE}\nu_\cE=\int_0^1\mu_\rho d\lambda(\rho).\end{equation}

\end{theorem}

	
\begin{remark}
 It may seem surprising that there is no stationary probability concentrated on the invariant sets $\cF',\cF'',\cE'$. This should be interpreted as a manifestation of the fact that abnormalities (e.g.\@ finitely many active particles in a mostly frozen configuration) escape to infinity in large time and cannot be seen in a stationary regime.
\end{remark}

	\section{Facts about invariant sets and connections}\label{s:Facts}
	We start by stating a few facts that follow from definitions \eqref{e:defF}--\eqref{e:defE}.

\begin{definition}
	Let $\Lambda$ interval of $\Z$, $\sigma,\sigma'\in\Omega_\Lambda$ two configurations. We say that $\sigma$ and $\sigma'$ are connected inside $\Lambda$, and we write $\sigma\overset{\Lambda}{\leftrightarrow}\sigma'$ if it is possible to reach $\sigma'$ from $\sigma$ by performing a finite number of allowed jumps inside $\Lambda$. When $\Lambda=\Z$ we omit the superscript.
	\end{definition}

\begin{lemma}
\label{l:connections}
We have the following facts.
\begin{enumerate}
	 \item $\cF$, the $\cF'_k$ with $k\geq 2$, the $\cF''_k$ with $k\in\N$, $\cE'$ and $\cE$ are disjoint and invariant sets for the dynamics with generator \eqref{e:defL} (in fact they are invariant under allowed jumps). 
	 
	 \item For any $\Lambda$ interval of $\Z$, $\sigma,\sigma'\in G_\Lambda$, $|\sigma|=|\sigma'|$ implies $\sigma\overset{\Lambda}{\leftrightarrow}\sigma'$.
	 
	 \item For all $\sigma,\sigma'\in\cF'_k$ (resp.\ $\sigma,\sigma'\in\cF''_k$), for all $n\in\N^*$ such that $\{x\in\Z\ \colon\ \sigma(x)+\sigma'(x)>0\}\subset\Lambda_{n}$ (resp. $\{x\in\Z\ \colon\ \sigma(x)\sigma'(x)=0\}\subset\Lambda_{n-2}$), $\sigma\overset{\Lambda_n}{\leftrightarrow}\sigma'$. 
	 
	 \end{enumerate}
	\end{lemma}
\begin{proof}The first point should be clear from the definitions. 

The second point is a crucial consequence of the existence of a mobile cluster inside $\Lambda$. If $|\sigma|=2$, this is a restatement of the fact that a mobile cluster can transport itself to any other location (see Figure~\ref{f:mobile-cluster}). For the case $|\sigma|\geqslant 3$, the crucial argument is that a mobile cluster can also take an additional particle along with it. It is enough to treat the case where all particles of $\sigma'$ are massed to the right of $\Lambda$. As suggested in Figure~\ref{f:mobile-cluster}, it is possible to connect $\sigma$ to a configuration with a particle at the right-most site of $\Lambda$ and a mobile cluster in the remaining part of $\Lambda$. We can then repeat this procedure until all particles are massed to the right.

The third point is a consequence of the second: if $\sigma,\sigma'\in\cF'_k$ (resp. $\cF''_k$), $n$ is chosen so that their restriction to $\Lambda_n$ is in $G_n$ and $|\sigma_{\Lambda_n}|=|\sigma'_{\Lambda_n}|$. Indeed, in the $\cF'$ case, $\Lambda_n$ contains all the particles of $\sigma,\sigma'$, and since these configurations are not frozen, at least one mobile cluster. In the $\cF''$ case, $\Lambda_{n-2}$ contains all the empty sites of $\sigma,\sigma'$, and therefore both configurations have a mobile cluster in the two right-most sites of $\Lambda_n$.
\end{proof}

Our first result identifies the configurations that appear with positive probability under $\nu$.
\begin{lemma}\label{l:positive}
	Fix $\nu$ an invariant measure for the PMM. 
	\begin{enumerate}
		\item For all $k\in\N\setminus\{0,1\}$, either $\nu(\cF'_k)=0$ or $\forall \eta\in\cF'_k$, $\nu(\eta)>0$. 
		
		For all $k\in\N$, either $\nu(\cF''_k)=0$ or $\forall \eta\in\cF''_k$, $\nu(\eta)>0$. 
		\item If $\nu(\cE\cup\cE')>0$, for all $n\in\N^*$, for all $\sigma\in\Omega_{\Lambda_n}$, $\nu(\sigma)>0$.
		\end{enumerate}
\end{lemma}

\begin{proof} The lemma relies on Lemma~\ref{l:connections} and the following observation: if $\nu$ is invariant, for all $\sigma\in\Omega_{\Lambda_n}$, $\nu(\cL \mathbf{1}_\sigma)=0$. This means 
	\[\sum_{x=-n-1}^n\left[\nu(c_x\mathbf{1}_\sigma^{x,x+1})-\nu(c_x\mathbf{1}_\sigma)\right]=0,\]
	where $f^{x,x+1}(\eta)=f(\eta^{x,x+1})$ for $\eta\in\Omega$. In particular, if $\nu(\sigma)=0$, the negative part vanishes in every term of the sum, and $\nu(c_x\mathbf{1}_\sigma^{x,x+1})=0$ for all $x\in\{-n-1,\ldots,n\}$. Equivalently, all configurations that can be obtained from $\sigma$ by an allowed jump have $0$ probability under $\nu$. Iterating the argument yield that all configurations $\sigma'$ connected to $\sigma$ inside $\Lambda_n$ also satisfy $\nu(\sigma')=0$.
	
	Let us now prove the two points of the lemma. 
	\begin{enumerate}
		\item Fix $k\in\N\setminus\{0,1\}$ and assume $\nu(\cF_k')>0$. Assume by contradiction that there exists $\eta\in\cF'_k$ such that $\nu(\eta)=0$. Let us show that for any $\eta'\in\cF'_k$, $\nu(\eta')=0$, which contradicts $\nu(\cF_k')>0$. Choosing $n$ large enough so that $\{x\in\Z\ \colon\ \eta(x)+\eta'(x)>0\}\subset\Lambda_{n}$, Lemma~\ref{l:connections} gives the result.
		
		The proof is similar for $\cF''_k$.
		\item Assume by contradiction that there exists $\sigma\in\Omega_{\Lambda_n}$ such that $\nu(\sigma)=0$. We claim that, if $\nu(\cE\cup\cE')>0$, there exists $N\in\N^*$, $\eta,\eta'\in\Omega_{\Lambda_N}$ such that $\eta_{|\Lambda_n}=\sigma$, $\nu(\eta')>0$ and $\eta\overset{\Lambda_N}{\leftrightarrow}\eta'$. This is a contradiction because $\nu(\sigma)=0$ implies $\nu(\eta)=0$, which implies by connection $\nu(\eta')=0$. 
		
		The claim follows if we can show that there exists $N\geq n+2$, $\eta'\in G_N$ such that $|\eta'|\geq |\sigma|+2$, $|1-\eta'|\geq |1-\sigma|$ and $\nu(\eta')>0$. Indeed, in that case, it is easy to extend $\sigma$ into a configuration $\eta\in G_N$ such that $|\eta|=|\eta'|$. Then, by Lemma~\ref{l:connections}, $\eta$ and $\eta'$ are connected inside $\Lambda_N$.
		
		Let us show that we can find $N,\eta'$ as above. By definition of $\cE,\cE'$, for any $\xi\in\cE\cup\cE'$, there exists $N\in\N^*$ such that $\xi_{|\Lambda_N}\in G_N$, $|\xi_{|\Lambda_N}|\geq |\sigma|+2$ and  $|1-\xi_{|\Lambda_N}|\geq |1-\sigma|$ (and these properties hold for any $N'\geq N$). Therefore, since $\nu(\cE\cup\cE')>0$, there exists $N\geq n+2$ such that 
		\[\nu(\{\xi\in\Omega\ \colon\ \xi_{|\Lambda_N}\in G_N, |\xi_{|\Lambda_N}|\geq |\sigma|+2\text{ and }|1-\xi_{|\Lambda_N}|\geq |1-\sigma|\})>0,\]
		which shows the existence of $\eta'$ as desired.
		\end{enumerate}
	\end{proof}

\section{Proof of Theorem~\ref{t:invariant}}\label{s:proofTheorem}

	From the invariance of the sets $\cF,\cF',\cF'',\cE,\cE'$ under allowed jumps, we deduce immediately the existence of a decomposition $
	\nu=\alpha_\cF\nu_{\cF}+\alpha_{\cF'}\nu_{\cF'}+\alpha_{\cF''}\nu_{\cF''}+\alpha_{\cE'}\nu_{\cE'}+\alpha_\cE\nu_\cE,
	$ with $\nu_X $ stationary, $\nu_X(X)=1$ and $\alpha_X=\nu(X)$ for $X\in\{\cF,\cF',\cF'',\cE,\cE'\}$, letting $\nu_X=\nu(\cdot|X)$ if $\nu(X)>0$. Indeed in that case, for any $f$ local function $\nu(\mathcal Lf|X)=\nu(\mathbf{1}_X\mathcal Lf)/\nu(X)=\nu(\cL (\mathbf{1}_Xf))/\nu(X)=0$, so that $\nu_X$ is stationary. 
	
	It remains to show $\alpha_{\cF'}=\alpha_{\cF''}=\alpha_{\cE'}=0$, and that $\nu_\cE$ has the form \eqref{e:decompnuE}. 
	The main ingredient to obtain these results is the following lemma.
\begin{lemma}\label{l:exchangeinvariant}
	If $\nu$ is stationary and $\nu(\cF'\cup\cF''\cup\cE\cup\cE')>0$, for all $\sigma\in\Omega_{\Lambda_n}$, $x\in\{-n,\ldots,n-1\}$,
	\[c_{x,x+1}(\sigma)\left[\nu(\sigma^{x,x+1})-\nu(\sigma)\right]=0.\]
	Consequently, for any $\Lambda\subset \Z$, $\sigma,\sigma'\in\Omega_\Lambda$, if $\sigma\overset{\Lambda}{\leftrightarrow}\sigma'$, then $\nu(\sigma)=\nu(\sigma')$.
	\end{lemma}
Let us defer the proof of Lemma~\ref{l:exchangeinvariant} to Section~\ref{s:Lemma} to conclude the proof of Theorem~\ref{t:invariant}. 

\medskip

Assume first that $\nu=\nu_\cE$. We want to show~\ref{e:decompnuE}. By de Finetti theorem, it is enough to show that $\nu$ is exchangeable, i.e.\@ that for $\sigma\in\Omega_{\Lambda_n}$, $\nu(\sigma)$ depends only on $|\sigma|$. Suppose $\Lambda=\Lambda_{n_0}$ and consider $\sigma,\sigma'\in\Omega_\Lambda$ such that $|\sigma|=|\sigma'|$. It is enough to show $\nu(\sigma)=\nu(\sigma')$ in that case. If $\sigma,\sigma'\in G_\Lambda$,  Lemma~\ref{l:connections} and the corollary of Lemma~\ref{l:exchangeinvariant} imply $\nu(\sigma)=\nu(\sigma')$. It remains to treat the case where $\sigma$ or $\sigma'$ is not in $G_\Lambda$. 

For $n\geq 1$, write $B_n=\Lambda_{n_0+n}\setminus\Lambda$. For $\sigma\in\Omega_{\Lambda}$, write
\begin{equation}\label{e:decompositionnusigma}
\nu(\sigma)=\sum_{n\geq 2}\sum_{\zeta\in G_{B_n}\setminus G_{B_{n-1}}}\nu(\sigma\cdot\zeta),
\end{equation}
where $\sigma\cdot\zeta$ denotes the configuration equal to $\sigma$ in $\Lambda$ and to $\zeta$ in $B_n$. This equality holds because, thanks to the definition of $\cE$, with probability $1$ under $\nu$, there exists a pair of active particles outside $\Lambda$. 

Now fix $\sigma,\sigma'\in\Omega_\Lambda$ such that $|\sigma|=|\sigma'|$, $\zeta\in G_{B_n}$. We clearly have $\sigma\cdot\zeta,\sigma'\cdot\zeta\in G_{\Lambda_{n_0+n}}$ and $|\sigma\cdot\zeta|=|\sigma'\cdot\zeta|$. Therefore, $\nu(\sigma'\cdot\zeta)=\nu(\sigma\cdot\zeta)$. Summing these equalities in the decompositions \eqref{e:decompositionnusigma} for $\sigma,\sigma'$ implies in turn $\nu(\sigma)=\nu(\sigma')$.

\medskip

Let us now consider the case $\nu=\nu_{\cF'}$ or $\nu=\nu_{\cF''}$. We want to show that there is a contradiction with $\nu$ being stationary. Fix $k$ such that e.g.\@ $\nu(\cF'_k)>0$. By Lemma~\ref{l:positive}, denoting $\chi_A$ the indicator function of a set $A$, $\nu(\{\eta= \chi_{\{1,\ldots,k\}}\})>0$. Fix now $n\in\N$ and an interval $\Lambda$ containing $\{1,\ldots,k\}\cup\{n+1,\ldots,n+k\}$. By Lemma~\ref{l:positive} and Lemma~\ref{l:exchangeinvariant}, $\nu(\{\eta= \chi_{\{n+1,\ldots,n+k\}}\})$ does not depend on $n$. But this means 
\begin{equation}
 \nu(\cF'_k)\geq\sum_{n\in\N}\nu(\{\eta= \chi_{\{n+1,\ldots,n+k\}}\})=\infty.
\end{equation}
With similar arguments in the case $\nu=\nu_{\cF''}$, we conclude that $\alpha_{\cF'}=\alpha_{\cF''}=0$.

\medskip

 It remains to find a similar contradiction when $\nu=\nu_{\cE'}$.  For $\Lambda\subset\Z$, consider the event
\begin{equation}
 A_\Lambda:=\{\eta\in\cE'\ :\ \text{all active particles of $\eta$ are in }\Lambda\}.
\end{equation}
Then by definition, $\cE'=\bigcup_{n=1}^\infty A_{\Lambda_n}$. In particular, there exists $n\in\N^*$ such that $\nu(A_{\Lambda_n})>0$. We may and do assume $n$ even.
Let us show that Lemma~\ref{l:exchangeinvariant} implies, for all $k\in\N$, $\nu(A_{\Lambda_n+3kn})=\nu(A_{\Lambda_n})>0$. This is a contradiction since the $\left(A_{\Lambda_n+3kn}\right)_k$ are disjoint (because $(\Lambda_{n}+3kn)\cap(\Lambda_n+3k'n)=\emptyset$ if $k\neq k'$) and $\nu(\cE')<\infty$. 
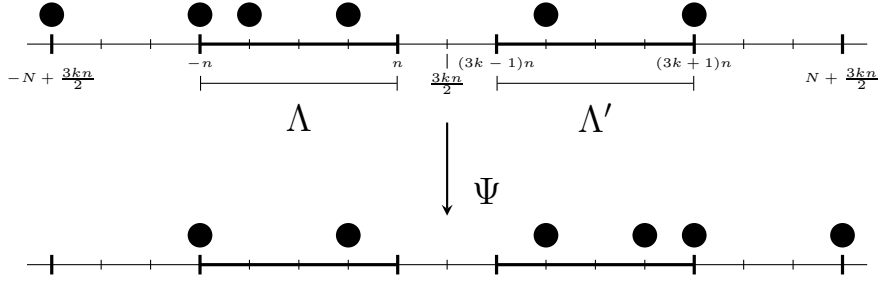
\begin{figure}
 \begin{center}
\begin{tikzpicture}[scale=0.65]
			\draw (-8.5,0) -- (8.5,0);
			\draw[very thick] (-5,0) -- (-1,0);
			\draw[very thick] (1,0) -- (5,0);
			\foreach \x in {-8,...,8}
			\draw (\x, -0.1) -- (\x,0.1);

			\draw[very thick] (-8,-0.2) -- (-8,0.2);
			\draw[very thick] (8,-0.2) -- (8,0.2);
			\draw[very thick] (-5,-0.2) -- (-5,0.2);
			\draw[very thick] (5,-0.2) -- (5,0.2);
			\draw[very thick] (-1,-0.2) -- (-1,0.2);
			\draw[very thick] (1,-0.2) -- (1,0.2);

            \draw[|-|] (-5,-0.8) to (-1,-0.8);
            \draw[|-|] (1,-0.8) to (5,-0.8);
			\draw (-3,-1.5) node{\Large{${\Lambda}$}};
			\draw (3,-1.5) node{\Large{${\Lambda'}$}};

			\draw (-8,-0.7) node{\tiny{$-N+\frac{3kn}{2}$}};
            \draw (8,-0.7) node{\tiny{$N+\frac{3kn}{2}$}};
            \draw (0,-0.8) node{\tiny{$\frac{3kn}{2}$}};
            \draw (0,-0.2) -- (0,-0.5);
            \draw (-5,-0.4) node{\tiny{$-n$}};
            \draw (-1,-0.4) node{\tiny{$n$}};
            \draw (5,-0.4) node{\tiny{$(3k+1)n$}};
            \draw (1,-0.4) node{\tiny{$(3k-1)n$}};

            \fill (-8,0.6) circle(0.25);
            \fill (-5,0.6) circle(0.25);
            \fill (-4,0.6) circle(0.25);
            \fill (-2,0.6) circle(0.25);
            \fill (2,0.6) circle(0.25);
            \fill (5,0.6) circle(0.25);
            
            \draw[-stealth, thick](0,-1.6) -- (0,-3.5);
            \draw (0.8,-3) node{\Large{$\Psi$}};
            
            \draw (-8.5,-4.5) -- (8.5,-4.5);
			\draw[very thick] (-5,-4.5) -- (-1,-4.5);
			\draw[very thick] (1,-4.5) -- (5,-4.5);
			\foreach \x in {-8,...,8}
			\draw (\x, -4.6) -- (\x,-4.4);

			\draw[very thick] (-8,-4.7) -- (-8,-4.3);
			\draw[very thick] (8,-4.7) -- (8,-4.3);
			\draw[very thick] (-5,-4.7) -- (-5,-4.3);
			\draw[very thick] (5,-4.7) -- (5,-4.3);
			\draw[very thick] (-1,-4.7) -- (-1,-4.3);
			\draw[very thick] (1,-4.7) -- (1,-4.3);
			
			   \fill (-5,-3.9) circle(0.25);
            \fill (5,-3.9) circle(0.25);
            \fill (4,-3.9) circle(0.25);
            \fill (2,-3.9) circle(0.25);
            \fill (-2,-3.9) circle(0.25);
            \fill (8,-3.9) circle(0.25);
            
			\end{tikzpicture}
			\end{center}
			\caption{On the first line are depicted $\Lambda,\Lambda'$ appearing in the proof that $\alpha_{\cE'}=0$, with $n=2,k=1$. The configuration on top is in $B_N$, and the bottom part of the picture gives its image under the application $\Phi$ defined in \eqref{e:Phi}.}\label{f:Lambdas}
			\end{figure}
Fix $k\in\N$ and let $\Lambda=\Lambda_n$, $\Lambda'=\Lambda_n+3kn$. For $N\geq (3k+1)N$, we let (see Figure~\ref{f:Lambdas}) $\Lambda_{N,1}=[-N,-n-1]\cap\Z$, $\Lambda_{N,2}=[n+1,N]$, $\Lambda'_{N,1}=[(3k+1)n+1,N]$, $\Lambda'_{N,2}$. We have that $A_{\Lambda}=\cap_{N\geq (3k+1)n}B_N$, $A_{\Lambda'}=\cap_{N\geq (3k+1)n}B'_N$, where
\begin{align*}
 B_N&:=\{\eta\in\Omega: \text{there are active particles in }\Lambda_N+(3kn)/2\text{ with empty boundary condition and they are in }\Lambda\},\\
 B'_N&:=\{\eta\in\Omega: \text{there are active particles in }\Lambda_N+(3kn)/2\text{ with empty boundary condition and they are in }\Lambda'\}.
\end{align*}
Note that $B_N,B'_N$ are not subsets of $\cE'$, and $\mathbf{1}_{B_N},\mathbf{1}_{B'_N}$ are local functions with support in $\Lambda_N$. Moreover, $\nu(A_\Lambda)=\lim_{N\rightarrow \infty}\nu(B_N)$ and $\nu(A_{\Lambda'})=\lim_{N\rightarrow \infty}\nu(B'_N)$ since the sequences $(B_N)_N,(B'_N)_N$ are decreasing. Therefore, it is enough to show that $\nu(B_N)=\nu(B'_N)$. This follows from Lemma~\ref{l:exchangeinvariant} and the following bijection between $B_N$ and $B'_N$:
\begin{equation}\Psi(\sigma)(-N+(3kn)/2+x)=\sigma(N+(3kn)/2-x),\text{ for all } x\in\Lambda_N+(3kn)/2.\label{e:Phi}
\end{equation}
In words, $\Psi$ performs a mirror reflection on $\sigma$ inside $\Lambda_N+(3kn)/2$ (see Figure~\ref{f:Lambdas}). It is easy to check that this defines indeed a bijection, and Lemma~\ref{l:exchangeinvariant} implies $\forall \sigma\in B_N$, $\nu(\sigma)=\nu(\Psi(\sigma))$. It follows that $\alpha_{\cE'}=0$.

\section{Proof of the invariance under allowed exchanges (Lemma 3)}\label{s:Lemma}
	The following is an adaptation of the arguments in \cite{holley-stroock}, as well as their reformulations in \cite[Section IV.4]{liggett} or \cite{neuhauser-sudbury}. 
	
	Let us first assume $\nu(\cE\cup\cE')>0$. Thanks to Lemma~\ref{l:positive}, in that case, the following quantity is well-defined for $\rho\in(0,1)$.
	\begin{equation}
	H_n(\nu)=\sum_{\sigma\in\Omega_{\Lambda_n}}\nu(\sigma)\log\left[\frac{\nu(\sigma)}{\mu_\rho(\sigma)}\right].
	\end{equation}
	Moreover, writing $\nu_t=\nu$ for the distribution of the process at time $t$ starting from $\nu$ and with generator~\eqref{e:defL}, $\frac{d}{dt}H_n(\nu_t)=0$. Computing the derivative and evaluating it at time $0$, we get the identity 
	\begin{equation}
	\sum_{\sigma\in\Omega_{\Lambda_n}}\log\left(\frac{\nu(\sigma)}{\mu_\rho(\sigma)}\right)\nu(\cL \mathbf{1}_{\sigma})=0.
	\end{equation}
	
For $n\in\N^*$, $x\in\{-n,\ldots,n-1\}$, $\sigma\in\Omega_{\Lambda_n}$, let us define
\begin{equation}
\Gamma_n(x,\sigma)=\nu(c_{x,x+1}\mathbf{1}_\sigma).
\end{equation}
Note that, for $x\in\{-n+1,\ldots,n-2\}$, we simply have $\Gamma_n(x,\sigma)=c_{x,x+1}(\sigma)\nu(\sigma)$. Also, for all $m\geq n$, 
\begin{equation}\label{e:Gammaadditive}
\Gamma_n(x,\sigma)=\sum_{\zeta\in\Omega_{\Lambda_m},\zeta_{|\Lambda_n}=\sigma}\Gamma_m(x,\zeta).
\end{equation}

	Let us split this into $I_n+B_n=0$, with
	\begin{align}
	I_n=\sum_{\sigma\in\Omega_{\Lambda_n}}\sum_{x=-n+1}^{n-2}\log&\left(\frac{\nu(\sigma)}{\mu_\rho(\sigma)}\right)\left[\Gamma_n(x,\sigma^{x,x+1})-\Gamma_n(x,\sigma)\right],\\
	B_n=\sum_{\sigma\in\Omega_{\Lambda_n}}\log\left(\frac{\nu(\sigma)}{\mu_\rho(\sigma)}\right)\Big\{&\Gamma_n(-n,\sigma^{-n,-n+1})-\Gamma_n(-n,\sigma)\\
	&
	+\ \Gamma_n(n-1,\sigma^{n-1,n})-\Gamma_n(n,\sigma)\\
	&+\sum_{\zeta\in\Omega_{\Lambda_{n+1}},\zeta_{|\Lambda_n}=\sigma}\left(\Gamma_{n+1}(-n-1,\zeta^{-n-1,-n})-\Gamma_{n+1}(-n-1,\zeta)\right)\\
	&+\sum_{\zeta\in\Omega_{\Lambda_{n+1}},\zeta_{|\Lambda_n}=\sigma}\left(\Gamma_{n+1}(n,\zeta^{n,n+1})-\Gamma_{n+1}(n,\zeta)\right)\Big\}.
	\end{align}
	
	$I_n$ contains the bulk terms, for which $c_{x,x+1}\mathbf{1}_\sigma$ only depends on the configuration inside $\Lambda_n$. In particular we can write $\Gamma_n(x,\sigma)=c_{x,x+1}(\sigma)\nu(\sigma)$ for $x\in\{-n+1,\ldots,n-2\}$. $B_n$ contains the boundary terms which need to look outside $\Lambda_n$ to decide the value of $c_{x,x+1}$ and, for the last two lines, that of $\sigma^{x,x+1}$.
	
	Using that $c_{x,x+1}(\sigma)=c_{x,x+1}(\sigma^{x,x+1})$, let us symmetrize $I_n$ into
	\begin{align}
	I_n&=-\frac{1}{2}\sum_{x=-n+1}^{n-2}\sum_{\sigma\in\Omega_{\Lambda_n}}\log\left(\frac{\nu(\sigma^{x,x+1})\mu_\rho(\sigma)}{\nu(\sigma)\mu_\rho(\sigma^{x,x+1})}\right)\left[\Gamma_n(x,\sigma^{x,x+1})-\Gamma_n(x,\sigma)\right]\\
	&=-\frac{1}{2}\sum_{x=-n+1}^{n-2}\sum_{\sigma\in\Omega_{\Lambda_n}}c_{x,x+1}(\sigma)\log\left(\frac{\nu(\sigma^{x,x+1})}{\nu(\sigma)}\right)(\nu(\sigma^{x,x+1})-\nu(\sigma)).
	\end{align}
	Define $\Phi(u,v)=\log(u/v)(v-u)$ and notice that it is convex, subadditive\footnote{Because convex and $\Phi(\lambda a,\lambda b)=\lambda\Phi(a,b)$. Indeed, by homogeneity $\Phi(a+a',b+b')=2\Phi((a+a')/2,(b+b')/2)\leq \Phi(a,b)+\Phi(a',b')$ by convexity.} and non-negative. In particular, defining 
	\begin{equation}
	\alpha_n(x)=\sum_{\sigma\in\Omega_{\Lambda_n}}c_{x,x+1}(\sigma)\Phi(\nu(\sigma^{x,x+1}),\nu(\sigma)),
	\end{equation}
	subadditivity implies that $(\alpha_n(x))_{n\geq N(x)}$ ($N(x)=\inf\{n\colon {x,x+1}\in\Lambda_n\}$) is a non-decreasing sequence of non-negative real numbers. Therefore, if $I_n\underset{n\rightarrow \infty}{\longrightarrow} 0$, $\alpha_n(x)=0$ for all $n\geq N(x)$. Since $\Phi(u,v)=0$ iff $u=v$, this is enough for our purposes and we only need to show that $B_n\underset{n\rightarrow \infty}{\longrightarrow} 0$.
	
	We first show that $\sup_nB_n<\infty$. Since $B_n=-I_n$ is non-negative, this implies that $(B_n)_{n}$ is bounded. Also, since \[-2I_n=\sum_{x=-n+1}^{n-2}\alpha_n(x)\]
	$-2(I_{n+1}-I_n)\geq \alpha_{n+1}(-n)\geq 0$, $(I_n)_n$ is non-increasing, and $(B_n)_n$ non-decreasing, so we will also deduce that $(B_n)_n$ converges. Let us symmetrize $B_n$:
	\begin{align}
	2B_n=&-\sum_{\sigma\in\Omega_{\Lambda_n}}\log\left(\frac{\nu(\sigma^{-n,-n+1})\mu_p(\sigma)}{\nu(\sigma)\mu_p(\sigma^{-n,-n+1})}\right)\left[\Gamma_n(-n,\sigma^{-n,-n+1})-\Gamma_n(-n,\sigma)\right]\label{e:Bn1}\\
	&- \sum_{\sigma\in\Omega_{\Lambda_n}}\log\left(\frac{\nu(\sigma^{n-1,n})\mu_p(\sigma)}{\nu(\sigma)\mu_p(\sigma^{n-1,n})}\right)\left[\Gamma_n(n-1,\sigma^{n-1,n})-\Gamma_n(n-1,\sigma)\right]\label{e:Bn2}\\
	&-\sum_{\sigma\in\Omega_{\Lambda_{n+1}}}\log\left(\frac{\nu(\sigma_{|\Lambda_n}^{-n-1,-n})\mu_p(\sigma_{|\Lambda_n})}{\nu(\sigma_{|\Lambda_n})\mu_p(\sigma_{|\Lambda_n}^{-n-1,-n})}\right)\left[\Gamma_{n+1}(-n-1,\sigma^{-n-1,-n})-\Gamma_{n+1}(-n-1,\sigma)\right]\label{e:Bn3}\\
	&-\sum_{\sigma\in\Omega_{\Lambda_{n+1}}}\log\left(\frac{\nu(\sigma_{|\Lambda_n}^{n,n+1})\mu_p(\sigma_{|\Lambda_n})}{\nu(\sigma_{|\Lambda_n})\mu_p(\sigma_{|\Lambda_n}^{n,n+1})}\right)\left[\Gamma_{n+1}(n,\sigma^{n,n+1})-\Gamma_{n+1}(n,\sigma)\right].\label{e:Bn4}
	\end{align}
	Let us bound \eqref{e:Bn1} and \eqref{e:Bn4}. $\eqref{e:Bn2}$ and $\eqref{e:Bn3}$ are bounded in similar ways. 
	\begin{align}
	\eqref{e:Bn1}=&-\sum_{\sigma\in\Omega_{\Lambda_n}}\log\left(\frac{\nu(\sigma^{-n,-n+1})}{\nu(\sigma)}\right)\left[\Gamma_n(-n,\sigma^{-n,-n+1})-\Gamma_n(-n,\sigma)\right]\\
	\overset{\eqref{e:Gammaadditive}}{=}&-\sum_{\sigma\in\Omega_{\Lambda_{n}}}\sum_{\zeta\in\Omega_{\Lambda_{n+1}},\zeta_{|\Lambda_n}=\sigma}\log\left(\frac{\nu(\sigma^{-n,-n+1})}{\nu(\sigma)}\right)\left[\Gamma_{n+1}(-n,\zeta^{-n,-n+1})-\Gamma_{n+1}(-n,\zeta)\right]\\
	=&\sum_{\sigma\in\Omega_{\Lambda_{n}}}\sum_{\zeta\in\Omega_{\Lambda_{n+1}},\zeta_{|\Lambda_n}=\sigma}\log\left(\frac{\nu(\sigma^{-n,-n+1})}{\nu(\sigma)}\right)c_{-n,-n+1}(\zeta)\left[\nu(\zeta)-\nu(\zeta^{-n,-n+1})\right].
	\end{align}
	Let us split the sum according to the value of $\frac{\nu(\sigma^{-n,-n+1})}{\nu(\sigma)}$. Let $K'\geq K\geq 1$. If $\frac{\nu(\sigma^{-n,-n+1})}{\nu(\sigma)}\in[K,K']$ and $\nu(\zeta^{-n,-n+1})\geq\nu(\zeta)$, we discard the corresponding term. If $\frac{\nu(\sigma^{-n,-n+1})}{\nu(\sigma)}\in[K,K']$ and $\nu(\zeta^{-n,-n+1})\leq\nu(\zeta)$, we discard $\nu(\zeta^{-n-n+1})$ and use $\sum_{\zeta\in\Omega_{\Lambda_{n+1}},\zeta_{\Lambda_n}=\sigma,\nu(\zeta^{-n,-n+1})\leq\nu(\zeta)}\nu(\zeta)\leq\nu(\sigma)$ to bound
	\[\sum_{\zeta\in\Omega_{\Lambda_{n+1}},\zeta_{|\Lambda_n}=\sigma}\log\left(\frac{\nu(\sigma^{-n,-n+1})}{\nu(\sigma)}\right)c_{-n,-n+1}(\zeta)\left[\nu(\zeta)-\nu(\zeta^{-n,-n+1})\right]\leq \frac{\log K'}{K}\nu(\sigma^{-n,-n+1}).\]
	Therefore, the sum over $\sigma$ such that $\frac{\nu(\sigma^{-n,-n+1})}{\nu(\sigma)}\in[K,K']$ is bounded by $\frac{\log K'}{K}$.
	
	Similarly, the sum over $\sigma$ such that $\frac{\nu(\sigma^{-n,-n+1})}{\nu(\sigma)}\in[\delta',\delta]$ with $0<\delta'<\delta\leq 1$ is bounded by $|\log(\delta')|\delta$. Let us now fix $K>1$ and define $K_k=Ke^k$, $\delta_k=K^{-1} e^{-k}$. We can deduce from the bounds we just established that
	\begin{align}
	\eqref{e:Bn1}\leq \log K\sum_{\zeta\in\Omega_{\Lambda_{n+1}}}c_{-n,-n+1}(\zeta)|\nu(\zeta)-\nu(\zeta^{-n,-n+1})|+2\sum_{k\geq 0}\frac{\log K +k+1}{Ke^k}<\infty.\label{e:boundK}
	\end{align}
	Let us now focus on \eqref{e:Bn4}. For $\sigma\in\Omega_{\Lambda_n}$, $\zeta\in\Omega_{\Lambda_{n+2}},\zeta_{|\Lambda_n}=\sigma$, denote $\sigma\underset{n}{\leftarrow}\zeta$ the configuration in $\Omega_{\Lambda_n}$ equal to $\sigma$ on $\Lambda_n\setminus\{n\}$ and $\zeta(n+1)$ on $n$.
	\begin{align}
	\eqref{e:Bn4}&=-\sum_{\sigma\in\Omega_{\Lambda_n}}\sum_{\zeta\in\Omega_{\Lambda_{n+2}},\zeta_{|\Lambda_n}=\sigma}\log\left(\frac{\nu(\sigma\underset{n}{\leftarrow}\zeta)\mu_p(\sigma)}{\nu(\sigma)\mu_p(\sigma\underset{n}{\leftarrow}\zeta)}\right)\left[\Gamma_{n+2}(n,\zeta^{n,n+1})-\Gamma_{n+2}(n,\zeta)\right]\\
	&=\sum_{\sigma\in\Omega_{\Lambda_n}}\sum_{\zeta\in\Omega_{\Lambda_{n+2}},\zeta_{|\Lambda_n}=\sigma}\left[\log\frac{\nu(\sigma\underset{n}{\leftarrow}\zeta)}{\nu(\sigma)}+\log\frac{\mu_p(\sigma)}{\mu_p(\sigma\underset{n}{\leftarrow}\zeta)}\right]
	c_{n,n+1}(\zeta)\left[\nu(\zeta)-\nu(\zeta^{n,n+1})\right].
	\end{align}
	$|\log\frac{\mu_p(\sigma)}{\mu_p(\sigma\underset{n}{\leftarrow}\zeta)}|$ is clearly  bounded by $\log\frac{p\vee (1-p)}{p\wedge(1-p)}$, which takes care of this extra term. Other than that, we perform almost the same analysis as for \eqref{e:Bn1}, but we further need to split the sum over $\zeta$ into $\zeta(n+1)=\sigma(n)$ and $\zeta(n+1)=1-\sigma(n)$. The first one does not contribute, and the second can be bounded by the analog of the middle term in \eqref{e:boundK}, and the conclusion $\eqref{e:Bn4}<\infty$ still holds.
	
	Let us collect that we have established 
	\begin{equation}
	\label{e:estimBn}
	B_n\leq 8\sum_{k\geq 0}\frac{\log K +k+1}{Ke^k}+ \log K\left[\beta_{n+1}(n-1)+\beta_{n+1}(-n)+\beta_{n+2}(n)+\beta_{n+2}(-n-1)\right],
	\end{equation}
	for any $K>1$, where $\beta_m(x)=\sum_{\zeta\in\Omega_{\Lambda_{m}}}c_{x,x+1}(\zeta)|\nu(\zeta)-\nu(\zeta^{x,x+1})|$ for $\{x,x+1\}\subset\Lambda_{m-1}$.
	
	We have shown that $(B_n)_n$, and therefore $(I_n)_n$, converges. Moreover, recall that $0\leq \alpha_{n+1}(-n)\leq -2(I_{n+1}-I_n)$. In particular, $\alpha_{n+1}(-n)$ vanishes as $n\rightarrow\infty$. Similarly, the same holds for $\alpha_{n+1}(n-1)$. It now remains to notice that the following analog of \cite[Lemma IV.5.8]{liggett} holds: for $x,m$ such that $\{x,x+1\}\subset\Lambda_{m}$,
	\begin{equation}
	\beta_{m+1}(x)\leq \sqrt{2\alpha_{m+1}(x)}.
	\end{equation}
	This relies on the fact that one can write $\beta_{m+1}(x)=M-m$, where $M=\sum_{\zeta\in\Omega_{\Lambda_m}}c_{x,x+1}(\zeta)\max(\nu(\zeta),\nu(\zeta^{x,x+1})),$ and $m$ is defined similarly with a $\min$ instead of a $\max$. One can then use $M-m\leq M\log(M/m)$, the subadditivity and homogeneity of $\Phi$, and the fact that $M\leq 2$ to conclude.
	Therefore the second term in \eqref{e:estimBn} vanishes. This concludes the proof in the case $\nu(\cE\cup\cE')>0$.
	
	Let us now assume that $\nu(\cF'_k)>0$ for some $k\geqslant 2$. Lemma~\ref{l:positive} then implies the following modified entropy is well-defined:
	\begin{equation}
	 \tilde H_n(\nu)=\sum_{\sigma\in\Omega_{\Lambda_n},|\sigma|\leqslant k}\nu(\sigma)\log\left[\frac{\nu(\sigma)}{\mu_\rho(\sigma)}\right].
	\end{equation}
	The rest of the proof can be carried as above, with minor points of attention: 
	\begin{itemize}
	 \item one needs to define $\Phi(0,0)=0$; this preserves non-negativity of $\Phi$, positivity outside $\{(u,u),u\in\R_+\}$, convexity, homogeneity and subadditivity.
	 \item with this convention, the sums over $\zeta$ appearing in the proof can be taken as sums over $\zeta$ with $|\zeta|\leqslant k$.
	\end{itemize}
	
	The case $\nu(\cF''_k)>0$ for some $k\geqslant 1$ is similar.
 
\bibliographystyle{plain} 
\bibliography{stationary-measures}
\end{document}